\documentclass[reqno]{amsproc}

\usepackage{graphicx}

\usepackage{hyperref}
\usepackage{epstopdf}
\hypersetup{
colorlinks=true,
linkcolor=blue,
anchorcolor=blue,
citecolor=blue
}

\usepackage{amssymb}
\usepackage{amsfonts}
\usepackage{amsmath}
\usepackage{dsfont, enumerate}
\newtheorem{theorem}{Theorem}[section]
\newtheorem{lemma}[theorem]{Lemma}
\newtheorem{proposition}[theorem]{Proposition}
\newtheorem{corollary}[theorem]{Corollary}

\theoremstyle{definition}
\newtheorem{definition}[theorem]{Definition}

\newtheorem{remark}[theorem]{Remark}

\numberwithin{equation}{section}

\newcommand{\abs}[1]{\lvert#1\rvert}

\newcommand{\arxiv}[1]{{\tt \href{http://arxiv.org/abs/#1}{arXiv:#1}}}
\newcommand{\set}[1]{\left\{ #1 \right\}}
\DeclareMathOperator*{\esssup}{ess\,sup}
\DeclareMathOperator*{\essinf}{ess\,inf}
\newcommand\restr[2]{{
		\left.\kern-\nulldelimiterspace 
		#1 
		\vphantom{\big|} 
		\right|_{#2} 
}} 

\begin{document}

\title{Geometric implications of fast volume growth and capacity estimates}

\author{Tim Jaschek}
\address{Department of Mathematics, University of British Columbia,
	Vancouver, BC V6T 1Z2, Canada.}
\curraddr{1QB Information Technologies (1QBit), Vancouver, BC V6E 4B1, Canada.}
\email{tim.jaschek@gmail.com }

\author{Mathav Murugan}
\address{Department of Mathematics, University of British Columbia,
	Vancouver, BC V6T 1Z2, Canada. }
\email{mathav@math.ubc.ca}
\thanks{Research partially supported by NSERC (Canada)}





\begin{abstract}
	We obtain connectivity of annuli for a volume doubling metric measure Dirichlet space which satisfies a Poincar\'e inequality, a capacity estimate and a fast volume growth condition. This type of connectivity was introduced by Grigor'yan and Saloff-Coste in order to obtain stability results for Harnack inequalities and to study diffusions on manifolds with ends. As an application of our result, we obtain   stability of the elliptic Harnack inequality under  perturbations of the Dirichlet form with radial type weights. 
\end{abstract}

\maketitle


\section*{Contents}
1. Introduction

2. Chain connectivity

3. From chain connectivity to path connectivity

4. Applications

References

\section{Introduction}
In this work, we study geometric consequences of analytic properties in the context of metric measure spaces equipped with a Dirichlet form. We are interested in connectivity properties of metric spaces at various scales and locations for spaces that satisfy a  Poincar\'e inequality, an upper bound on capacity, and certain conditions on volume growth. We obtain a connectivity condition on annuli introduced by Grigor'yan and Saloff-Coste \cite{GS1}. A similar connectivity of annuli was used to obtain heat kernel bounds on a family of planar graphs in \cite[Theorem 6.2(d)]{Mur-QS}.

Much of the motivation for our work arises from analysis and probability on fractals. For a large class of
fractal spaces  $(X,d)$, there exists a diffusion process  which
is symmetric with respect to some canonical measure $m$ and exhibits strong sub-diffusive
behavior in the sense that its transition density (heat kernel) $p_{t}(x,y)$ satisfies
the following \emph{sub-Gaussian estimate}:
\begin{equation}\label{e:HKEbeta}
\begin{split}
p_{t}(x,y) &\geq \frac{c_{1}}{m(B(x,t^{1/\beta}))} \exp\biggl(  - c_{2} \Bigl( \frac{d(x,y)^{\beta}}{t} \Bigr)^{\frac{1}{\beta-1}}\biggr),
\\
p_t(x,y) &\leq\frac{c_{3}}{m(B(x,t^{1/\beta}))} \exp\biggl(  - c_{4} \Bigl( \frac{d(x,y)^{\beta}}{t} \Bigr)^{\frac{1}{\beta-1}}\biggr),
\end{split}
\end{equation}
for all points $x,y \in X$ and all $t>0$, where $c_{1},c_{2},c_{3},c_{4}>0$ are some constants,
$d$ is a natural metric on $X$, $B(x,r)$ denotes the open ball of radius
$r$ centered at $x$, and $\beta\geq 2$ is an exponent describing the diffusion called the
\emph{walk dimension}. 
Often, $m$ is a Hausdorff measure  and is Ahlfors $d_f$-regular; that is $m(B(x,r)) \asymp r^{d_f}$ for all $x \in X$ and  $0<r<\operatorname{diam}(X,d)$. The number $d_f$ is called the \emph{volume growth exponent} of the space.
This result was obtained first for the Sierpi\'{n}ski gasket in
\cite{BP}, then for nested fractals in \cite{Kum}, for affine nested fractals in \cite{FHK}
and for Sierpi\'{n}ski carpets in \cite{BB99}. We refer to \cite{Bar98} for a general introduction to diffusions on fractals.

An important motivation for this work arises from a conjecture of Grigor'yan, Hu and Lau  \cite[Conjecture 4.15]{GHL14}, \cite[p. 1495]{GHL} -- see also \cite[Open Problem III]{Kum14}. The conjecture is a characterization of the sub-Gaussian heat kernel estimate \eqref{e:HKEbeta}, in terms of the volume doubling property, a capacity upper bound and a Poincar\'e inequality.
The answer to this conjecture is known only in certain ``low dimensional settings'' (or strongly recurrent case) \cite{BCK}. Recent progress had been made on a family of planar graphs \cite{Mur-QS} and on some transient graphs \cite{Mur-RC} but still under quite restricted assumptions. If we further assume that the measure $m$ is Ahlfors $d_f$-regular, then the setting in \cite{BCK} corresponds to $d_f<\beta$, where $\beta$ is the walk dimension as described above. 
Roughly speaking, we consider spaces that are complementary to the ``strongly recurrent'' regime considered in \cite{BCK}. In the case of polynomial volume growth as described above, our ``fast volume growth'' condition corresponds to the complementary case $d_f \ge \beta$ while \cite{BCK} considers $d_f < \beta$ -- see Definition \ref{d:fvg}.  Since this is the case where \cite[Conjecture 4.15]{GHL14} is still open, we hope our work will simulate further progress on the conjecture of Grigor'yan, Hu and Lau.

We briefly survey some previous related works.
A major milestone in the understanding of heat kernel bounds and Harnack inequalities is the characterization of the parabolic Harnack inequality by the combination of the volume doubling property and the Poincar\'e inequality  due to Grigor'yan and Saloff-Coste \cite{Gri,Sal}. Such a characterization implies the stability of the parabolic Harnack inequality under bounded perturbation of the Dirichlet form.
More recently, the understanding of geometric consequences of analytic properties has played an important role in works on the stability of elliptic Harnack inequality and the singularity of energy measures. In particular, a crucial step in obtaining the stability of elliptic Harnack inequality in \cite{BM} is the result that any geodesic space that satisfies the elliptic Harnack inequality admits a doubling measure. In \cite{KM}, the chain condition plays a role in the proof of the singularity of the energy measures for spaces satisfying the sub-Gaussian heat kernel estimate.  In  \cite{Mur}, the chain condition was obtained as a consequence of the sub-Gaussian heat kernel estimate.

\subsection{Main results}

Throughout this paper, we consider a complete, locally compact separable metric space $(X,d)$,
equipped with a Radon measure $m$ with full support, that is, a Borel measure $m$ on $X$
which is finite on any compact set and strictly positive on any non-empty open set.
Such a triple $(X,d,m)$ is referred to as a \emph{metric measure space}. In what follows, we set
$\operatorname{diam}(X,d):=\sup_{x,y\in X}d(x,y)$ and $B(x,r):=\{y\in X\mid d(x,y)<r\}$ for $x\in X$ and $r>0$.

Let $(\mathcal{E},\mathcal{F})$ be a \emph{symmetric Dirichlet form} on $L^{2}(X,m)$.
In other words, the domain $\mathcal{F}$ is a dense linear subspace of $L^{2}(X,m)$, such that
$\mathcal{E}:\mathcal{F}\times\mathcal{F}\to\mathbb{R}$
is a non-negative definite symmetric bilinear form which is \emph{closed}
($\mathcal{F}$ is a Hilbert space under the inner product $\mathcal{E}_{1}(\cdot,\cdot):= \mathcal{E}(\cdot,\cdot)+ \langle \cdot,\cdot \rangle_{L^{2}(X,m)}$)
and \emph{Markovian} (the unit contraction operates on $\mathcal{F}$, that is, $(u \vee 0)\wedge 1\in\mathcal{F}$ and $\mathcal{E}((u \vee 0)\wedge 1,(u \vee 0)\wedge 1)\leq \mathcal{E}(u,u)$ for any $u\in\mathcal{F}$).
Recall that $(\mathcal{E},\mathcal{F})$ is called \emph{regular} if
$\mathcal{F}\cap\mathcal{C}_{\mathrm{c}}(X)$ is dense both in $(\mathcal{F},\mathcal{E}_{1})$
and in $(\mathcal{C}_{\mathrm{c}}(X),\|\cdot\|_{\mathrm{sup}})$.
Here $\mathcal{C}_{\mathrm{c}}(X)$ is the space of $\mathbb{R}$-valued continuous functions on $X$
with compact support.

For a function $u \in \mathcal{F}$, let $\operatorname{supp}_{m}[u] \subset X$ denote the support of the measure $|u|\,dm$, that is, the smallest closed subset $F$ of $X$ with $\int_{X\setminus F}|u|\,dm=0$. Note that
$\operatorname{supp}_{m}[u]$ coincides with the closure of $X\setminus u^{-1}(\{0\})$ in $X$ if $u$ is continuous.
Recall that
$(\mathcal{E},\mathcal{F})$ is called \emph{strongly local}
if
$\mathcal{E}(u,v)=0$ for any $u,v\in\mathcal{F}$ with $\operatorname{supp}_{m}[u]$, $\operatorname{supp}_{m}[v]$
compact and $v$ is constant $m$-almost everywhere in a neighborhood of $\operatorname{supp}_m[u]$.
The pair $(X,d,m,\mathcal{E},\mathcal{F})$ of a metric measure space $(X,d,m)$ and a strongly local,
regular symmetric Dirichlet form $(\mathcal{E},\mathcal{F})$ on $L^{2}(X,m)$ is termed a \emph{metric measure Dirichlet space}, or a \emph{MMD space}.
We refer to \cite{FOT,CF} for a comprehensive account of the theory of symmetric Dirichlet forms.

We recall the notion of curves and path connectedness  in a metric space.
\begin{definition}[Path connected]{\rm Let $(X,d)$ be a metric space and let $A \subset X$.
		We say that $\gamma$ is a \emph{curve} in $A$ from $x$ to $y$ if $\gamma:[0,1] \to A$ is continuous,  $\gamma(0)=x$ and $\gamma(1)=y$.  
		For two sets $B_1 \subset B_2 \subset X$ we say that $B_1$ is \emph{path connected} in $B_2$ if for all $x,y \in B_1$, there exists a curve in $B_2$ from $x$ to $y$.
}\end{definition}

Henceforth, we fix a function $\Psi:(0,\infty) \to (0,\infty)$ to be a continuous increasing bijection of $(0,\infty)$ onto itself, such that for all $0 < r \le R$,
\begin{equation}  \label{e:reg}
C^{-1} \left( \frac R r \right)^{\beta_1} \le \frac{\Psi(R)}{\Psi(r)} \le C \left( \frac R r \right)^{\beta_2}, 
\end{equation}
for some constants $0 < \beta_1 < \beta_2$ and $C>1$. Throughout this work, the function $\Psi$ is meant to denote the space time scaling of a symmetric diffusion process.  

\begin{definition}[Volume doubling]
	{\rm
		We say that $(X,d,m)$ satisfies the \emph{volume doubling property} \hypertarget{vd}{$\operatorname{(VD)}$} if there exists $C_D \ge 1$ such that
		\begin{equation}\tag*{($\operatorname{VD}$)} 
		m(B(x,2r)) \le C_D m(B(x,r)), \quad \mbox{ for all $x \in X, r >0$.}
		\end{equation}}
\end{definition}

We recall the definition of energy measures associated to a MMD space. Note that $fg\in\mathcal{F}$
for any $f,g\in\mathcal{F}\cap L^{\infty}(X,m)$ by \cite[Theorem 1.4.2-(ii)]{FOT}
and that $\{(-n)\vee(f\wedge n)\}_{n=1}^{\infty}\subset\mathcal{F}$ and
$\lim_{n\to\infty}(-n)\vee(f\wedge n)=f$ in norm in $(\mathcal{F},\mathcal{E}_{1})$
by \cite[Theorem 1.4.2-(iii)]{FOT}.

Let $(X,d,m,\mathcal{E},\mathcal{F})$ be a MMD space.
The \emph{energy measure} $\Gamma(f,f)$ of $f\in\mathcal{F}$ is defined,
first for $f\in\mathcal{F}\cap L^{\infty}(X,m)$ as the unique ($[0,\infty]$-valued)
Borel measure on $X$ with the property that
\begin{equation}\label{e:EnergyMeas}
\int_{X} g \, d\Gamma(f,f)= \mathcal{E}(f,fg)-\frac{1}{2}\mathcal{E}(f^{2},g), \qquad \mbox{ for all $g \in \mathcal{F}\cap\mathcal{C}_{\mathrm{c}}(X)$,}
\end{equation}
and then by
$\Gamma(f,f)(A):=\lim_{n\to\infty}\Gamma\bigl((-n)\vee(f\wedge n),(-n)\vee(f\wedge n)\bigr)(A)$
for each Borel subset $A$ of $X$ for general $f\in\mathcal{F}$. 

The notion of an energy measure can be extended to the \emph{local Dirichlet space} $\mathcal{F}_{\operatorname{loc}}$, which is defined as
\begin{equation}\label{e:Floc}
\mathcal{F}_{\operatorname{loc}} := \biggl\{ f \in L^2_{\operatorname{loc}}(X,m) \biggm|
\begin{minipage}{230pt}
For any relatively compact open subset $V$ of $X$, there exists
$f^{\#} \in \mathcal{F}$ such that $f \mathds{1}_{V} = f^{\#} \mathds{1}_{V}$ $m$-a.e..
\end{minipage}
\biggr\}.
\end{equation}
For any $f \in \mathcal{F}_{\operatorname{loc}}$ and for any relatively compact open set $V \subset X$, we define
\[
\Gamma(f,f)(V) := \Gamma(f^{\#},f^{\#})(V),
\]	
where $f^{\#}$ is as in the definition of $\mathcal{F}_{\operatorname{loc}}$. Since $(\mathcal{E},\mathcal{F})$ is strongly local, the value of $\Gamma(f^{\#},f^{\#})(V)$ does not depend on the choice of $f^{\#}$, and is therefore well defined. Since $X$ is locally compact, this defines a Radon measure $\Gamma(f,f)$ on $X$.

\begin{definition}[Poincar\'e inequality]
	{\rm
		We say that $(X,d,m,\mathcal{E},\mathcal{F})$ satisfies the \emph{Poincar\'e inequality} \hypertarget{pi}{$\operatorname{PI}(\Psi)$}, if there exist constants $C_P,A  \ge 1$ such that 
		for all $x\in X$, $r \in (0, \infty)$ and $f \in \mathcal{F}_{\operatorname{loc}}$
		\begin{equation} \tag*{$\operatorname{PI}(\Psi)$}
		\int_{B(x,r)} (f - \overline f)^2 \,dm  \le C_P \Psi(r) \, \int_{B(x,A r)}d\Gamma(f,f),
		\end{equation}
		where $\overline f= m(B(x,r))^{-1} \int_{B(x,r)} f\, dm$.
}\end{definition}
The following elementary observation will be used along with the Poincar\'e inequality:
\begin{equation}\label{e:var} 
\int_{B(x,r)} (f - \overline f)^2 \,dm = \frac{1}{2 m(B(x,r))} \int_{B(x,r)} \int_{B(x,r)} \abs{f(y)-f(z)}^2 \, m(dy)\,m(dz).
\end{equation}

\begin{definition}[Capacity estimate]
	{\rm
		
		Let   $(X,d,m,\mathcal{E},\mathcal{F})$  be a MMD space.	
		For disjoint subsets $A,B \subset X$, we define
		\[
		\mathcal{F}(A,B):= \set{f \in \mathcal{F}:  
 \begin{minipage}{200pt}	$f \equiv 1$ on a neighborhood of $\overline A$ and  $f \equiv 0$ on a neighborhood of $\overline B$ \end{minipage} } ,
		\]
		and the \emph{capacity} $\operatorname{Cap}(A,B)$ as 
		\[
		\operatorname{Cap}(A,B) := \inf \set{\mathcal E(f,f): f \in \mathcal F(A,B)}.
		\]
		
		We say that  $(X,d,m,\mathcal{E},\mathcal{F})$ satisfies the \emph{capacity estimate} \hypertarget{cap}{$\operatorname{cap}(\Psi)_\le$} 	
		if there exist $C_1,A_1>1$ such that for all $0< r <\operatorname{diam}(X,d)/A_1$, $x \in X$ 
		\begin{equation}\tag*{$\operatorname{cap}(\Psi)_\le$}
		\operatorname{Cap}(B(x,r),B(x,A_1r)^c) \le C_1 \frac{m(B(x,r))}{\Psi(r)}. 
		\end{equation} 
	}
\end{definition}

\begin{definition}[Fast volume growth] \label{d:fvg}{\rm
		We say that $(X,d,m,\mathcal{E},\mathcal{F})$ satisfies the \emph{fast volume growth condition} \hypertarget{fvg}{$\textup{FVG}(\Psi)$}  if there exists a constant $C_F > 0$ such that 
		\begin{equation}
		\tag*{\ensuremath{\textup{FVG}(\Psi)}}
		\label{FVG}
		\frac{\Psi(R)}{\Psi(r)} \le  C_F \frac{m(B(x,R))}{m(B(x,r))},
		\end{equation}
		for all $0<r \le  R < \operatorname{diam}(X,d)$ and $x \in X$.}
\end{definition}

Our main result is the following path connectedness of annuli.
\begin{theorem} \label{t:annuluscurve}
	Let $(X,d,m,\mathcal{E},\mathcal{F})$ be a MMD space that satisfies  \hyperlink{vd}{$\operatorname{(VD)}$}, \hyperlink{pi}{$\operatorname{PI}(\Psi)$}, \hyperlink{cap}{$\operatorname{cap}(\Psi)_\le$} and \hyperlink{fvg}{$\textup{FVG}(\Psi)$}, where $\Psi$ satisfies \eqref{e:reg}.
	Then there exists $C_0 \ge 2$ such that for all $x \in X, r>0$, $B(x,2r) \setminus B(x,r)$ is path connected in $B(x,C_0 r) \setminus B(x,r/C_0)$.
\end{theorem}

\begin{remark} \label{r:main}
	{\rm
		\begin{enumerate}
			\item  The condition \hyperlink{fvg}{$\textup{FVG}(\Psi)$} is not necessary for the conclusion of Theorem \ref{t:annuluscurve}. For example, the Brownian motion on the standard two-dimensional Sierpi\'{n}ski carpet satisfies all of the hypotheses except \hyperlink{fvg}{$\textup{FVG}(\Psi)$} and satisfies the conclusion.  
			On the other hand, if the volume growth exponent $d_f$ is strictly less than the walk dimension $\beta$, the gluing construction of Delmotte \cite{Del} shows that \hyperlink{fvg}{$\textup{FVG}(\Psi)$} is a sharp condition. In particular, if  $d_f<\beta$, then by gluing two copies of the same space at a point, one obtains a space that satisfies \hyperlink{vd}{$\operatorname{(VD)}$}, \hyperlink{pi}{$\operatorname{PI}(\Psi)$}, \hyperlink{cap}{$\operatorname{cap}(\Psi)_\le$} with $\Psi(r)=r^\beta$ but fails to satisfy the conclusion.
			\item 
			Our argument is quite flexible and can be localized at different scales. For example, the cylinder $\mathbb{S} \times \mathbb{R}$ satisfies the hypotheses and conclusion of Theorem \ref{t:annuluscurve} only at small enough scales. On the other hand, the cable system (graph with edges represented by copies of the unit interval) corresponding to $\mathbb{Z}^d, d \ge 2$, satisfies the hypotheses and conclusion only at large enough scales.
			\item Theorem \ref{t:annuluscurve} can be applied to obtain stability of elliptic Harnack inequality under perturbations of the Dirichlet form (Theorem \ref{t:weight}). Furthermore, our theorem can be combined with \cite[Theorem 1.1]{Mac} to obtain that the conformal dimension of $(X,d)$ is strictly greater than one.  
		\end{enumerate}	
		
	}
\end{remark}
Much of the proof involves estimating capacities between sets from above and below. Our proof is motivated by the arguments in \cite{Kor,HeKo,Mur}.
The basic idea behind our approach is that if the capacity between two sets is strictly positive, then the two sets cannot belong to different connected components. Using lower bounds on capacities implied by the Poincar\'e inequality, we obtain a connectivity result in Proposition \ref{p:llc1}. This along with upper bounds on capacity leads to a more quantitative estimate in Theorem \ref{t:chain}. Theorem \ref{t:chain} strengthens a recent result in \cite{Mur} which was used to show the chain condition for spaces satisfying sub-Gaussian heat kernel bounds. In this work, Theorem \ref{t:chain} plays a crucial role in upgrading from connectivity to path connectivity. 

\section{Chain connectivity}

We recall the definition of an $\epsilon$-chain in a metric space $(X,d)$.
\begin{definition}[$\epsilon$-chain]{\rm 
		Let $B \subset X$.
		We say that a sequence  $\set{x_i}_{i=0}^N$  of points in $X$ is an \emph{$\epsilon$-chain } in $B$ between points $x,y \in X$ if 
		\[
		x_0=x, \quad x_N=y, \quad \mbox{ and } \quad d(x_i,x_{i+1}) \le \epsilon \quad \mbox{ for all $i=0,1,\ldots,N-1$,}
		\]
		and 
		\[
		x_i \in B, \quad \mbox{ for all $i=0,1,\ldots,N$.}
		\]
		For any $\epsilon>0, B \subset X$ and $x,y \in B$, define
		\[
		N_\epsilon(x,y; B) = \inf \set{ n : {\set{x_i}_{i=0}^{n} \mbox{ is $\epsilon$-chain in $B$ between $x$ and $y$, $n \in \mathbb{N}$} } },
		\] 
		with the usual convention that $\inf \emptyset = 
		+\infty$. 
}\end{definition}
We introduce a notion of connectedness based on the existence of $\epsilon$-chains.
\begin{definition}[Chain connected]{\rm 
		Let $(X,d)$ be a metric space and let $B_1 \subseteq B_2 \subseteq X$. 
		We say that $B_1$ is \emph{chain connected} in $B_2$  if 
		\[
		N_\epsilon(x,y;B_2)<\infty, \quad \mbox{ for all $x,y \in B_1$ and for all $\epsilon>0$.} 
		\]
}\end{definition}
\begin{proposition} \label{p:llc1} Let $(X,d,m,\mathcal{E},\mathcal{F})$ be a MMD such that balls are precompact. Assume that $(X,d,m,\mathcal{E},\mathcal{F})$ satisfies the Poincar\'e inequality \hyperlink{pi}{$\operatorname{PI}(\Psi)$}. Then  for all $x \in X$, $r>0$, $B(x,r)$ is chain connected in $B(x,Ar)$, where $A \ge 1$ is the constant in  \hyperlink{pi}{$\operatorname{PI}(\Psi)$}.
\end{proposition}
\begin{remark} {\rm
		We note if $(X,d)$ is a complete metric space and if $m$ is a Borel measure that satisfies  \hyperlink{vd}{$\operatorname{(VD)}$}, all metric balls in $(X,d)$ are precompact. We always apply Proposition \ref{p:llc1}, when the metric space is complete and admits a doubling measure.
	}
\end{remark}
\noindent \emph{Proof of Proposition \ref{p:llc1}}. 
Let $A \ge 1$ be the constant in \hyperlink{pi}{$\operatorname{PI}(\Psi)$}.

Let $x \in X, \epsilon >0$ and  $U_x= \set{y \in B(x,Ar): N_\epsilon(x,y; B(x,Ar))< \infty}$.

By definition  of $U_x$, we have
\begin{equation}\label{e:cn1}
d(y,w) > \epsilon, \quad \mbox{for all $y \in U_x, w \in B(x,Ar) \setminus U_x$.}
\end{equation}
Let $N$ denote an $\epsilon/2$-net in $(X,d)$ such that $N \cap B(x,Ar)$ is an $\epsilon/2$-net of $B(x,Ar)$. By the precompactness of metric balls, the set $N \cap B(x,Ar)$ is  finite.
For each $z \in N$, choose a function $\phi_z \in C_c(X) \cap \mathcal{F}$ such that $1 \ge \phi_z \ge 0, \restr{\phi_z}{B(z,\epsilon/2)} \equiv 1$, and $\operatorname{supp}(\phi_z) \subset B(z,\epsilon)$. 
Define
\[
\phi(y)= \sup_{z \in N \cap U_x} \phi_z(y), \quad \mbox{ for all $y \in X$}.
\]
Since $N \cap U_x$ is a finite set, $\phi \in \mathcal{F}$.
By \eqref{e:cn1}, and since $\cup_{z \in N} B(z,\epsilon/2) = X$, we obtain
\begin{equation}\label{e:cn2}
\phi \equiv 1_{U_x}, \quad \mbox{ on $B(x,Ar)$}.
\end{equation}
By \cite[Theorem 4.3.8]{CF}, the push-forward measure of $\Gamma(\phi,\phi)$ by $\phi$ 
is absolutely continuous with respect to the $1$-dimensional Lebesgue measure.
Since $\set{0,1}$ has zero Lebesgue measure, by using \eqref{e:cn2} we obtain
\begin{equation} \label{e:cn3}
\Gamma(\phi,\phi)(B(x,Ar))= 0.
\end{equation}
Since $\phi$ is continuous, by \hyperlink{pi}{$\operatorname{PI}(\Psi)$}, $\phi$ is constant on the ball $B(x,r)$. Therefore $\phi(y)= \phi(x)=1$ for all $y \in B(x,r), r>0$. Therefore $\phi \equiv 1$ on $B(x,r)$, which along with \eqref{e:cn2} implies that $U_x \cap B(x,r) = B(x,r)$. 

Since  $\epsilon>0$ was arbitrary, we obtain that $B(x,r)$ is chain connected in $B(x,Ar)$ for all $x \in X, r>0$.
\qed

The following lemma records an useful consequence of the conditions \hyperlink{cap}{$\operatorname{cap}(\Psi)_\le$}  and \hyperlink{fvg}{$\textup{FVG}(\Psi)$}.
\begin{lemma} \label{l:fvg}
	Let $(X,d,m,\mathcal{E},\mathcal{F})$ be a MMD space that satisfies \hyperlink{vd}{$\operatorname{(VD)}$}, \hyperlink{cap}{$\operatorname{cap}(\Psi)_\le$}  and \hyperlink{fvg}{$\textup{FVG}(\Psi)$}, where $\Psi$ satisfies \eqref{e:reg}. Then for any $\delta>0$, there exist $C,A_\delta>1$ such that
	\begin{equation}
	\operatorname{Cap}(B(x,R/C),B(x,R)^c) \le \delta \frac{m(B(x,R))}{\Psi(R)},
	\end{equation}
	for all $x \in X, 0<R< \operatorname{diam}(X,d)/A_\delta$.
\end{lemma}
\begin{proof}
	Let $C_1, A_1$ be the constants in \hyperlink{cap}{$\operatorname{cap}(\Psi)_\le$}, and let $C_F$ denote the constant in \hyperlink{fvg}{$\textup{FVG}(\Psi)$}. 
	
	Therefore, for all $x \in X, 0<r \le R < \operatorname{diam}(X,d)/A_1$, we have 
	\begin{equation}\label{e:fv1}
	\operatorname{Cap}(B(x,r),B(x,A_1 r)^c) \le C_1 \frac{m(B(x,r))}{\Psi(r)} \le C_1 C_F\frac{m(B(x,R))}{\Psi(R)}.
	\end{equation}
	Therefore, by the strong locality (see \cite[Lemma 2.5]{GNY}) for any $k \in \mathbb{N}$
	\begin{align} \label{e:fv2}
	\operatorname{Cap}(B(x,A_1^{-k}R),B(x,R)^c) & \le \left( \sum_{i=0}^{k-1} \operatorname{Cap}(B(x,A_1^{-i-1}R),B(x,A_1^{-i}R)^c) ^{-1}  \right)^{-1} \nonumber \\
	& \le \left( \sum_{i=0}^{k-1} \frac{\Psi(R)}{C_1 C_F m(B(x,R))} \right)^{-1} \quad \mbox{(by \eqref{e:fv1})} \nonumber \\
	& \le \frac{C_1 C_F}{k} \frac{m(B(x,R))}{\Psi(R)}.
	\end{align}
	By choosing $k \in \mathbb{N}$ large enough so that $k > \delta^{-1} C_1 C_F$, we obtain the desired estimate. 
\end{proof}

In the following proposition, we obtain the chain connectivity of annuli under the same assumption as in our main result in Theorem \ref{t:annuluscurve}.
\begin{proposition} \label{p:chain}
	Let $(X,d,m,\mathcal{E},\mathcal{F})$ be a MMD space that satisfies  \hyperlink{vd}{$\operatorname{(VD)}$}, \hyperlink{pi}{$\operatorname{PI}(\Psi)$}, \hyperlink{cap}{$\operatorname{cap}(\Psi)_\le$} and \hyperlink{fvg}{$\textup{FVG}(\Psi)$}, where $\Psi$ satisfies \eqref{e:reg}.
	Then there exist $C,K \ge 2$ such that for all $x \in X, 0<r <\operatorname{diam}(X,d)/K$, $B(x,2r) \setminus B(x,r)$ is chain connected in $B(x,Cr) \setminus B(x,r/C)$.
\end{proposition}
\begin{proof}
	Let $A \ge 1$ denote the constant in \hyperlink{pi}{$\operatorname{PI}(\Psi)$}.
	Let $C\ge 3A$ be a constant whose value will be determined later in the proof.
	
	Assume by contradiction that $B(x,2r) \setminus B(x,r)$ is not chain connected in 
	$B(x,Cr) \setminus B(x,r/C)$. Then there exist $y,z \in B(x,2r)\setminus B(x,r)$ and $\epsilon>0$ such that  \[ N_\epsilon(y,z; B(x,Cr) \setminus B(x,r/C)) = \infty. \]

	Define
	\[
	U_y = \set { w \in B(x,Cr)\setminus B(x,r/C): N_\epsilon(w,y;  B(x,Cr)\setminus B(x,r/C))<\infty},
	\]
	and 
	\[
	V_y= \left( B(x,Cr)\setminus B(x,r/C)\right) \setminus U_y.
	\]
	By our assumption $z \notin U_y$. By Proposition \ref{p:llc1}, $B(z,r/(2A))$ is chain connected in $B(z,r/2)$. Using this and $B(z,r/2) \subset B(x,Cr)\setminus B(x,r/C))$,  we have 
	\begin{equation}\label{e:mn1}
	N_\epsilon(w,z;  B(x,Cr)\setminus B(x,r/C)) \le N_\epsilon( w, z; B(z,r/2)) < \infty, \mbox{for all $w \in B(z,r/(2A))$}.
	\end{equation}
	Combining \eqref{e:mn1} with $z \notin U_y$, we have
	\begin{equation}\label{e:mn2}
	B(z,r/(2A)) \cap U_y = \emptyset.
	\end{equation}
	By the same argument as \eqref{e:mn1},  we have
	\begin{equation} \label{e:mn3}
	B(y,r/(2A)) \subset U_y.
	\end{equation}
	
	Let $N$ be an $\epsilon/3$-net of $(X,d)$ such that $N \cap U_y \cap \left( B(x,Cr) \setminus B(x,r/C) \right)$ is an $\epsilon/3$-net of $U_y \cap B(x,Cr) \setminus B(x,r/C)$.
	For each $z \in N$, choose a function $\psi_z \in C_c(X) \cap \mathcal{F}$ such that 
	$0 \le \psi_z \le 1, \restr{\psi_z}{B(z,\epsilon/3)} \equiv 1$, and $\operatorname{supp}(\psi_z) \subset B(z,2\epsilon/3)$. Define
	\[
	\psi = \sup_{z \in N \cap U_y} \psi_z.
	\]
	We observe that
	\begin{equation} \label{e:mn5}
	\psi \in \mathcal{F} \cap C_c(X), \quad 0 \le \psi \le 1,  \quad \restr{\psi}{B(x,Cr)\setminus B(x,r/C)}\equiv  \restr{1_{U_y}}{B(x,Cr)\setminus B(x,r/C)}.
	\end{equation}
	Similar to \eqref{e:cn3}, we deduce that
	\begin{equation}\label{e:mn6}
	\Gamma(\psi,\psi)\left(B(x,Cr)\setminus B(x,r/C)\right) =0.
	\end{equation}
	Let $h \in C_c(B(x,r/2)) \cap \mathcal{F}$ be such that $0 \le h \le 1$, $\restr{h}{B(x, 2r/C)} \equiv 1$, and
	\begin{equation} \label{e:mn7}
	\mathcal{E}(h,h) \le 2 \operatorname{Cap}(B(x,2r/C),B(x,r/2)^c).
	\end{equation}
	Let $g= \max(h,\psi) \in C_c(X) \cap \mathcal{F}$. By \eqref{e:mn5} and the properties of $h$ above, we have
	\begin{equation}\label{e:mn8} 
	g \equiv 1 \quad \mbox{on $B(x,2r/C) \cup U_y$}, \qquad g \equiv 0 \quad \mbox{on $B(x,r/2)^c \cap U_y^c \cap B(x,Cr)$.}
	\end{equation}
	Furthermore, since $d(z_1,z_2) \ge \epsilon$ for all $z_1 \in U_y$ and $z_2 \in V_y$, we have $\psi \equiv 0$ on $\cup_{w \in V_y} B(w,\epsilon/3)$. Therefore $g \equiv h$ on $B(x,2r/C) \bigcup \cup_{w \in V_y} B(w,\epsilon/3)$.
	Hence
	\begin{equation}\label{e:mn9}
	g \equiv h, \quad \mbox{ in a neighborhood of $B(x,r/2) \setminus (B(x,2r/C) \cup U_y)$.}
	\end{equation}
	By \eqref{e:mn6} and strong locality 
	\begin{align} \label{e:mn10}
	\Gamma(g,g)(B(x,3Ar)) &= \Gamma(g,g)(B(x,r/2) \setminus (B(x,2r/C) \cup U_y))  \quad \mbox{(by \eqref{e:mn8})} \nonumber\\
	&= \Gamma(h,h)(B(x,r/2) \setminus (B(x,2r/C) \cup U_y))  \quad \mbox{(by \eqref{e:mn9})} \nonumber \\
	&\le \mathcal{E}(h,h) \quad \mbox{(by \eqref{e:mn5})} \nonumber\\
	& \le  2  \operatorname{Cap}(B(x,2r/C),B(x,r/2)^c)  \quad \mbox{(by \eqref{e:mn6}).}
	\end{align}
	Let $\delta>0$ be an arbitrary constant.
	By \eqref{e:reg}, and Lemma \ref{l:fvg}, we can choose $C\ge 3A$ large enough so that
	\begin{equation}\label{e:mn11}
	\operatorname{Cap}(B(x,2r/C),B(x,r/2)^c) \le \delta \frac{m(B(x,r/2))}{\Psi(r/2)} \le C_2 \delta \frac{m(B(x,r))}{\Psi(r)},
	\end{equation}
	where $C_2$ depends only the constants in \eqref{e:reg}.
	Combining \eqref{e:mn10} and \eqref{e:mn11}, we obtain
	\begin{equation} \label{e:mub}
	\Gamma(g,g)(B(x,3Ar)) \le 2 C_2 \delta \frac{m(B(x,r))}{\Psi(r)}.
	\end{equation}
	Evidently, by \eqref{e:mn2}, \eqref{e:mn3} and \eqref{e:mn8}, $g=\max(h,\psi)$  satisfies 
	\begin{equation}\label{e:mn4}
	g \equiv 1 \mbox{ on $B(y,r/(2A))$}, \quad	 g \equiv 0 \mbox{ on $B(z,r/(2A))$}. 
	\end{equation}
	By the triangle inequality, $B(y,r/(2A)) \cup B(z,r/(2A)) \subset B(x,3r)$. Now, we use the Poincar\'e inequality \hyperlink{pi}{$\operatorname{PI}(\Psi)$} along with \eqref{e:var} to derive the following estimate. Since $g \in \mathcal{F} \cap C(X)$ and satisfies \eqref{e:mn4}, we have
	\begin{align} \label{e:mlb}
	\lefteqn{\Gamma(g,g)\left(B(x,3Ar)\right)} \nonumber
	\\ &\ge \frac{1}{2 C_P m(B(x,3r))\Psi(3r)} \int_{B(x,3r)} \int_{B(x,3r)} \abs{g(p)-g(q)}^2 \, m(dp) \,m(dq) \nonumber \\
	&\ge \frac{1}{2 C_P m(B(x,3r))\Psi(3r)} \int_{B(z,r/(2A))} \int_{B(y,r/(2A))} \abs{g(p)-g(q)}^2 \, m(dp) \,m(dq) \nonumber \\
	&\ge \frac{m(B(y,r/(2A)) m(B(z,r/(2A))}{2 C_P  m(B(x,3r))\Psi(3r)}  \quad \mbox{(using \eqref{e:mn4})} \nonumber \\
	&\ge \frac{m(B(x,r))}{C_1 \Psi(r)}, \quad \mbox{(by \eqref{e:reg} and \hyperlink{vd}{$\operatorname{(VD)}$})}
	\end{align}
	where the constant $C_1 \ge 1$ depends only on the constants in \eqref{e:reg}, \hyperlink{vd}{$\operatorname{(VD)}$}, and \hyperlink{pi}{$\operatorname{PI}(\Psi)$}.
	

	By choosing $\delta< (2 C_1 C_2)^{-1}$, the bounds \eqref{e:mub} and \eqref{e:mlb} lead to the desired contradiction.
\end{proof}

\section{From chain connectivity to path connectivity}
The following result strengthens a bound on the length of chains given in \cite{Mur}.
The improvement is that the upper  bound was on $N_\epsilon(x,y; X)$ instead of $N_\epsilon(x,y; B(x,A_0 d(x,y))$.
\begin{theorem} \label{t:chain}
	Let $(X,d,m,\mathcal{E},\mathcal{F})$ be a MMD space that satisfies  \hyperlink{vd}{$\operatorname{(VD)}$}, \hyperlink{pi}{$\operatorname{PI}(\Psi)$}, and \hyperlink{cap}{$\operatorname{cap}(\Psi)_\le$}, where $\Psi$ satisfies \eqref{e:reg}.
	Then there exist $C,A_0>1$ such that for all $\epsilon>0$ and for all $x,y \in X$ that satisfy $d(x,y) \ge \epsilon$, we have
	\begin{equation}\label{e:chain}
	N_\epsilon(x,y; B(x,A_0 d(x,y)))^2 \le C \frac{\Psi(d(x,y))}{	\Psi(\epsilon) }.
	\end{equation}
\end{theorem}

For two measures $m, \nu$ on $(X,d)$, for $R> 0, x \in X$, we define a \emph{truncated maximal function}
\begin{equation}
M_R^m \nu(x)= \sup_{0< r < R} \frac{\nu(B(x,r))}{m(B(x,r))}.
\end{equation}
If $\nu \ll m$, then the above expression is the truncated maximal function of the Radon-Nikodym derivative $\frac{d\nu}{dm}$. However, in the lemma below $\nu$ will be the energy measure, and hence the measure $\nu$ and $m$ might be mutually singular \cite{KM}. In the following lemma, $\mathcal{C}(X)$ denotes the space of continuous functions on $X$.
We recall two lemmas from \cite{Mur}.
\begin{lemma}[Two point estimate] \label{l:poin} (see \cite[Lemma 5.15]{HeKo} and \cite[Lemma 2.4]{Mur})
	Let $(X,d,m,\mathcal{E},\mathcal{F})$ be a MMD space that satisfies \hyperlink{vd}{$\operatorname{(VD)}$}, \hyperlink{pi}{$\operatorname{PI}(\Psi)$}, where $\Psi$ satisfies \eqref{e:reg}.
	There exists $C_P>1$ such that for all $x_0 \in X, R>0$, $x, y \in B(x_0, C_P^{-1} R)$, and for all $u \in \mathcal{C}(X) \cap \mathcal{F}_{\operatorname{loc}}$
	\[
	\abs{u(x)-u(y)}^2 \le C \Psi(R) \left(M_R^m \Gamma(u,u)(x)+ M_R^m \Gamma(u,u)(y)\right),
	\]
	where $\Gamma(u,u)$ denotes the energy measure of $u$.
\end{lemma} 

\begin{lemma}[Partition of unity] \cite[Lemma 2.5]{Mur} \label{l:partition}
	Let $(X,d,m,\mathcal{E},\mathcal{F})$ be a MMD space that satisfies \hyperlink{vd}{$(\operatorname{VD})$}, and \hyperlink{cap}{$\operatorname{cap}(\Psi)_\le$}. Let $\epsilon>0$ and let $V$ denote any $\epsilon$-net. Let $\epsilon<\operatorname{diam}(X,d)/A_1$, where $A_1 \ge 1$ is the constant in  \hyperlink{cap}{$\operatorname{cap}(\Psi)_\le$}.
	Then, there exists a family of functions $\set{\psi_z: z\in V}$ that satisfies the following properties:
	\begin{enumerate}[(a)]
		\item  $\set{\psi_z: z\in V}$ is partition of unity, that is,  $\sum_{z \in V} \psi_z  \equiv 1$.
		\item For all $z \in V$, $\psi_z \in C_c(X) \cap \mathcal{F}$ with  $0 \le \psi_z \le 1$, $\restr{\psi_z}{B(z,\epsilon/4)} \equiv 1$, and $\restr{\psi_z}{B(z,5\epsilon/4)^c} \equiv 0$.
		\item For all $z \in V$, $z' \in V \setminus \set{z}$, we have $\restr{\psi_{z'}}{B(z,\epsilon/4)} \equiv 0$. 
		\item There exists $C>1$ such that for all $z\in V$,
		\[
		\mathcal{E}(\psi_z,\psi_z) \le C \frac{m(B(z,\epsilon))}{\Psi(\epsilon)}.
		\]
	\end{enumerate}
\end{lemma}

\noindent{\em Proof of Theorem \ref{t:chain}.} 
Let $A_1$ denote the constant in  \hyperlink{cap}{$\operatorname{cap}(\Psi)_\le$}.
Since $N_\epsilon(x,y;B) \le N_{\epsilon'}(x,y;B)$ whenever $B \subset X$ and $\epsilon' \le \epsilon$, by replacing $\epsilon$ by $\epsilon/(2A_1)$ if necessary and by using \eqref{e:reg}, we assume that $\epsilon< \operatorname{diam}(X,d)/A_1$.

Fix $x,y \in X, \epsilon>0$ such that $d(x,y) \ge \epsilon$. Set $\epsilon'=\epsilon/3$. Let $V$ be an $\epsilon'$-net such that $\set{x,y} \subset V$. 
Let $A_0>0$ be
\begin{equation}\label{e:chain0}
A_0 = 2A (C_P+2),
\end{equation}
where $A$ is the constant in \hyperlink{pi}{$\operatorname{PI}(\Psi)$} and $C_P$ is as given in Lemma \ref{l:poin}.
Define $\hat{u}: V \cap B(x,2 (C_P+2) d(x,y)) \to [0,\infty)$ as 
\begin{equation}\label{e:chain1}
\hat{u}(z):= N_\epsilon(x,z; B(x,A_0 d(x,y)) ).
\end{equation}
We set $\widetilde V=  V \cap B(x,2(C_P+2) d(x,y))$.
By Proposition \ref{p:llc1}, $\hat{u}$ is finite. By definition,
\begin{equation}\label{e:chain2}
\abs{\hat{u}(z_1)- \hat{u}(z_2)} \le 1, \quad \mbox{ for all $z_1,z_2 \in \widetilde V$ such that $d(z_1,z_2)<\epsilon$.}
\end{equation}
Let   $\set{\psi_z: z\in V}$  denote the partition of unity defined in Lemma \ref{l:partition}.
Define  $u:X \to [0,\infty)$ as
\[
u(p):= \sum_{z \in \widetilde V} \hat u(z)\psi_z(p). 
\]
For any ball $B(x_0,r), x_0 \in X, r >0$, by Lemma \ref{l:partition}(b) we have
\begin{equation}\label{e:chain3}
u(p) = \sum_{z \in V \cap B(x_0, r+ 5\epsilon'/4)}  \hat u(z)\psi_z(p),\quad \mbox{ for all $p \in B(x_0,r)$.}
\end{equation}
Since $V \cap B(x_0, r+ 5\epsilon'/4)$ is a finite set by \hyperlink{vd}{$\operatorname{(VD)}$}, we obtain that $u \in \mathcal{F}_{\operatorname{loc}}$.
By Lemma \ref{l:partition}(b), we have $\restr{u}{B(z,\epsilon'/4)}\equiv \hat u(z)$ for all $z \in V$. Therefore, by \cite[Theorem 4.3.8]{CF}, 
the push-forward measure of $\Gamma(u,u)$ by $u$ is absolutely continuous with respect to the $1$-dimensional Lebesgue measure.
Therefore, we obtain
\begin{equation}\label{e:chain4}
\Gamma(u,u)(B(z,\epsilon'/4)) =0, \quad \mbox{ for all $z \in V$.} 
\end{equation}
By \eqref{e:chain3} and Lemma \ref{l:partition}(a), we have
\begin{equation}\label{e:chain5}
u(p)= \hat u (z)+ \sum_{w \in V \cap B(z, 9\epsilon'/4)}  (\hat u(w)- \hat u(z))\psi_w(p),\quad \mbox{ for all $p \in B(z,\epsilon'), z \in V$.}
\end{equation}
By \hyperlink{vd}{$\operatorname{(VD)}$}, there exits $C_1>1$ such that $\sup_{z \in V} \abs{V \cap B(z, 9\epsilon'/4)} \le C_1$. By \eqref{e:chain5}, and the Cauchy-Schwarz inequality, there exists $C_2>1$ such that the following holds: for all $z \in V \cap B(x,2 (C_P+1)d(x,y))$, we have
\begin{align} \label{e:chain6}
\Gamma(u,u)(B(z,\epsilon')) &\le C_1 \sum_{w \in V \cap B(z, 9\epsilon'/4)} (\hat u(w)- \hat u(z))^2 \mathcal{E}(\psi_w,\psi_w) \nonumber \\
&\lesssim \sum_{w \in V \cap B(z, 9\epsilon'/4)}  \frac{m(B(w,\epsilon'))}{\Psi(\epsilon')} \quad \mbox{ (by \eqref{e:chain2} and Lemma \ref{l:partition}(d))} \nonumber \\
&\le C_2 \frac{m(B(z,\epsilon'/2))} {\Psi(\epsilon)} \quad \mbox{ (by \hyperlink{vd}{$\operatorname{(VD)}$} and \eqref{e:reg})}.
\end{align}
In the second line above, we used the fact that  $z \in V \cap  B(x,2 (C_P+1)d(x,y))$ and $w \in  V \cap B(z, 9\epsilon'/4)$ implies $w \in  B(x,2 (C_P+2)d(x,y))$.
By Lemma \ref{l:partition}, \eqref{e:chain6} and \hyperlink{vd}{$\operatorname{(VD)}$}, there exists $C_3>0$ such that for all $z \in V \cap B(x,2d(x,y)), r \le 2 C_P d(x,y)$, we have
\begin{equation}\label{e:chain7}
\Gamma(u,u)(B(z,r)) \le  C_2 \sum_{w \in B(z,r+5\epsilon'/4)}  \frac{m(B(w,\epsilon'/2))} {\Psi(\epsilon)} \le C_3 \frac{m(B(z,r))}{\Psi(\epsilon)}.
\end{equation}
Combining \eqref{e:chain4} and \eqref{e:chain7}, we obtain
\begin{equation}\label{e:chain8}
M^m_R\Gamma(u,u)(z)= \sup_{r <R} \frac{\Gamma(u,u)(B(z,r))}{m(B(z,r))}\le  \frac{C_3}{\Psi(\epsilon)},
\end{equation} 
for all $z \in \set{x,y}, 0<R \le 2 C_Pd(x,y)$.
By \eqref{e:chain8}, Lemma \ref{l:poin}, $\hat u (x)=0, \hat u(y)= N_\epsilon(x,y; B(x,A_0 d(x,y)))$, and \eqref{e:reg}, there exists $C_4 >0$ such that
\[ 
N_\epsilon(x,y; B(x,A_0 d(x,y)))^2 \le C_4 \frac{\Psi(d(x,y))}{\Psi(\epsilon)} \mbox{ for all $x,y \in X , \epsilon \le d(x,y)$.}
\]
Thus we obtain \eqref{e:chain}. 
\qed

\begin{corollary} \label{c:curve}	Let $(X,d,m,\mathcal{E},\mathcal{F})$ be a MMD space which satisfies  \hyperlink{vd}{$\operatorname{(VD)}$}, \hyperlink{pi}{$\operatorname{PI}(\Psi)$}, and \hyperlink{cap}{$\operatorname{cap}(\Psi)_\le$}, where $\Psi$ satisfies \eqref{e:reg}.
	Then there exists a curve from $x$ to $y$ in $B(x,2 A_0 d(x,y))$ for all $x,y \in X$, where $A_0$ is the constant in Theorem \ref{t:chain}.
\end{corollary}
\begin{proof}
	By \eqref{e:reg} and Theorem \ref{t:chain}, for any $\epsilon \in (0,1/2)$, there exists $N \in \mathbb{N}$ such that
	\[
	N_{\epsilon d(x,y)} (x,y, B(x, A_0 d(x,y))) \le N, \quad \mbox{for all $x,y \in X$,}
	\]
	where $A_0$ is the constant in Theorem \ref{t:chain}. For the remainder of the proof, we fix $\epsilon \in (0,1/2)$ and $N \in \mathbb{N}$ as above.
	
	Let $x,y \in X$ a pair of distinct points. For each $k \in \mathbb{N}$, we define $\gamma_k:[0,1] \to X$ as follows. Let $z^{(1)}_0, z^{(1)}_1,\ldots, z^{(1)}_N$ 
	be a sequence of points in $B(x, A_0 d(x,y))$ such that $d(z^{(1)}_i,z^{(1)}_{i+1})< \epsilon d(x,y),$ with $z^{(1)}_0=x, z^{(1)}_N=y$. Let $\gamma_1:[0,1] \to X$ be the piecewise constant function on intervals defined by
	\[
	\gamma_1(t)= z^{(1)}_i, \quad \mbox{for all $i=0,\ldots,N-1$ and for all $i/(N+1)\le t < (i+1)/(N+1)$}
	\]
	and $\gamma_1(1)= y$. Similarly, for all $i=0,\ldots, N$ and we chose 
	$z^{(2)}_j, j=i(N+1), i(N+1)+1,\ldots, i(N+1)+N$ such that 
	$z^{(2)}_{i(N+1)}=z^{(1)}_i, z^{(2)}_{i(N+1)+N}=z^{(1)}_{i+1}, z^{(2)}_k \in B(z^{(1)}_i, A_0d(z^{(1)}_i,z^{(1)}_{i+1}))$, for $k= i(N+1),\ldots, i(N+1)+N$, $d(z^{(2)}_j, z^{(2)}_{j+1}) < \epsilon^2 d(x,y)$  and define 
	\[
	\gamma_2(t)= z^{(2)}_j,
	\]
	for all $j=0,1,\ldots,(N+1)^2-1$ and for all $j/(N+1)^2\le t < (j+1)/(N+1)^2$
	with $\gamma_2(1)=y$. We similarly define $\gamma_k:[0,1] \to X$ that is piecewise constant on intervals $[j/(N+1)^k, (j+1)/(N+1)^k)$, $j=0,1,\ldots, (N+1)^k -1$. 
	Since for all $t \in [0,1]$, $d(\gamma_k(t), \gamma_{k+1}(t)) < A_0 \epsilon^k d(x,y)$, the sequence $\set{\gamma_k(t), k \in \mathbb{N}}$ is Cauchy, and hence converges to say $\gamma(t) \in X$. 
	This limit defines a function $\gamma:[0,1] \to X$.
	Note that
	\[
	d(x,\gamma(t)) \le \sum_{k=0}^\infty  \epsilon^k d(x,y)= A_0 d(x,y)/(1-\epsilon) <2 A_0 d(x,y).
	\]
	If $ \abs{t_1 - t_2} \le \frac{1}{N^k}$ for some $k \in \mathbb{N}$, we have
	\begin{align*}
	d(\gamma(t_1),\gamma(t_2)) &\le d(\gamma_k(t_1), \gamma(t_1)) + d(\gamma_k(t_2), \gamma(t_2))+ d(\gamma_k(t_1),\gamma_k(t_2)) \\
	& \le 2 \left( \sum_{l=k}^\infty A_0 \epsilon^l d(x,y) \right) + 2 (A_0+1) \epsilon^k d(x,y)\\
	&\le 2(A_0+1) (1+ (1-\epsilon)^{-1})\epsilon^k d(x,y),
	\end{align*}
	which implies the continuity of $\gamma$. Hence $\gamma: [0,1] \to B(x,2A_0d(x,y))$ is continuous and $\gamma(0)=x, \gamma(1)=y$.
\end{proof}

We now have all tools at our disposal to prove the main result.\\
\noindent{\em Proof of Theorem \ref{t:annuluscurve}.} 
We choose $C_0=2C$, where $C$ is the constant in Proposition \ref{p:chain}. Let $\epsilon=\frac{r}{4 C A}$, where $A$ is the constant in \hyperlink{pi}{$\operatorname{PI}(\Psi)$}.  By Proposition \ref{p:chain}, for any $x \in X,r >0$ and $y,z \in B(x,2r) \setminus B(x,r)$ 
\[
N_\epsilon(y,z; B(x,Cr) \setminus B(x,r/C))<\infty.
\]
Let $\set{x_i}_{i=1}^N$ be an $\epsilon$-chain between $y$ and $z$ in $B(x,Cr) \setminus B(x,r/C)$.
By Corollary \ref{c:curve}, there exist curves $\gamma_i:[0,1] \to B(z_i,r/(2C)), i=0,\ldots,N-1$ from $z_i$ to $z_{i+1}$. Since $B(z_i,r/(2C)) \subset  B(x,C_0 r) \setminus B(x,r/C_0)$ for all $i=0,\ldots,N-1$, by concatenating the curves $\gamma_i,i=0,\ldots,N-1$ we obtain a curve from $y$ to $z$ in $B(x,C_0 r) \setminus B(x,r/C_0)$.
\qed

Under the assumptions of Theorem \ref{t:annuluscurve}, we can obtain a quantitative bound on $N_\epsilon$ as follows.
By the volume doubling property, the minimum number of balls of radii $r$ required to cover a ball of radius $R$ depends only $R/r$. Using this property, we could obtain an uniform bound (that does not depend on $x$ or $r$) on $N_{r/(4CA)}(y,z,B(x,Cr) \setminus B(x,r/C))$, in the proof of Theorem \ref{t:annuluscurve}. As a consequence, using Theorem \ref{t:chain}, for any $\epsilon >0, x \in X$, $0<r<\operatorname{diam}(X,d)/A$, we have
\[
N_{r/(4CA)}(y,z,B(x,Cr) \setminus B(x,r/C))^2 \lesssim \frac{\Psi(r)}{\Psi(\epsilon)}, \quad \mbox{for all $y,z \in B(x,2r) \setminus B(x,r)$.}
\]

\section{Applications}
For simplicity, we will assume that our MMD space is unbounded; that is, $\operatorname{diam}(X,d)=\infty$.
We recall a condition introduced in \cite{GS1,GS2}. 
\begin{definition}[Relatively connected annuli]
	{\rm
		We say that $(X,d)$ satisfies (RCA) (this stands for
		\emph{relatively connected annuli}) if there is a point $o$ and a constant $A$ such
		that for any $r > A$ and any two points $x, y \in X$ with $d(o, x) = d(o, y) =
		r$ there is a continuous path in $B (o, Ar) \setminus B (o, r /A)$ connecting $x$ to $y$. 
	}
\end{definition}

A weighted manifold is a Riemannian manifold $(M,g)$ equipped with a measure $\mu$ that has a smooth positive density with respect to the Riemannian measure. This space is equipped with a weighted Laplace operator that generalizes the Laplace-Beltrami operator and is symmetric with respect to the measure $\mu$. Such spaces naturally arise in the context of Doob $h$-transforms and Schr\"odinger operators. We refer the reader to \cite{Gri-wt} for a comprehensive survey on weighted manifolds and applications.

Motivated by weighted manifolds (and the corresponding weighted Laplace operator), we perturb a MMD space by a continuous function $w :X \to (0,\infty)$  by a family of admissible weights.
\begin{definition}[Admissible weight] \label{d:admiss}{
		{\rm
			Let $(X,d,m,\mathcal{E},\mathcal{F})$ be a MMD space and let $w:X \to (0,\infty)$ be continuous.
			We say that $w$ is an \emph{admissible weight} if the following conditions hold:
			\begin{enumerate}[(a)]
				\item There exist $o \in X$, $\alpha_1 ,\alpha_2 \in \mathbb{R}, C>1$ such that
				\begin{equation}\label{e:adm}
				C^{-1} \left( \frac{d(o,x)}{d(o,y)}\right)^{\alpha_1}  \frac{w(x)}{w(y)} \le C \left( \frac{d(o,x)}{d(o,y)}\right)^{\alpha_2}, 
				\end{equation} 
				for all $x,y \in X$ such that $d(o,y) \ge d(o,x)$.
				\item There exists a MMD space $(X,d, w\,dm, \mathcal{E}_w,\mathcal{F}_w)$, where $\mathcal{F}_w \subset \mathcal{F}_{\operatorname{loc}}$ and 
				\[
				\mathcal{E}_w(f,f)= \int_X w \, d\Gamma(f,f), \quad \mbox{for all $f \in \mathcal{F}_w$,}
				\]
				where $\mathcal{F}_{\operatorname{loc}}$ is as defined in \eqref{e:Floc}. Furthermore, $\mathcal{F} \cap \mathcal{C}_c(X)$ forms a core for the Dirichlet form $(\mathcal{E}_w,\mathcal{F}_w)$ on $L^2(w\,dm)$.
			\end{enumerate}
	}}
\end{definition}
In the context of manifolds, we refer to  \cite{GS2,Gri-wt} for examples of admissible weights.

We recall the definition of a generalized capacity bound introduced in \cite{GHL} based on a similar condition due to Andres and Barlow.
\begin{definition}[Generalized capacity estimate] \label{d:gcap} {\rm
		For open subsets $U,V$ of $X$ with $\overline{U} \subset V$, we say that
		a function $\varphi \in \mathcal{F}$ is a \emph{cutoff function} for $U \subset V$
		if $0 \le \varphi \le 1$, $\varphi=1$ on a neighborhood of $\overline{U}$ and $\operatorname{supp}_{m}[\varphi] \subset V$.
		Then we say that $(X,d,m,\mathcal{E},\mathcal{F})$ satisfies the
		\emph{Generalized capacity estimate} \hyperlink{Gcap}{$\operatorname{Gcap}(\Psi)$},
		if there exists $C_{S}>0$ such that the following holds:
		for each $x \in X$ and each $R,r>0, f \in \mathcal{F}$ there exists a cutoff function $\varphi \in \mathcal{F}$
		for $B(x,R) \subset B(x,R+r)$ such that  
		\begin{equation}\tag*{$\operatorname{Gcap}(\Psi)$}
		\int_{X} f^{2}\, d\Gamma(\varphi,\varphi)
		\leq C_S \int_{B(x,R+r) \setminus B(x,R)} \varphi^{2} \, d\Gamma(f,f)
		+ \frac{C_{S}}{\Psi(r)} \int_{B(x,R+r) \setminus B(x,R)} f^{2}\,dm.
		\end{equation}
		Here and in what follows, we always consider a quasi-continuous version of
		$f\in\mathcal{F}$, which exists by \cite[Theorem 2.1.3]{FOT} and is unique
		$\mathcal{E}$-q.e.\ (i.e., up to sets of capacity zero) by \cite[Lemma 2.1.4]{FOT},
		so that the values of $f$ are uniquely determined $\Gamma(g,g)$-a.e.\ for each $g\in\mathcal{F}$
		since $\Gamma(g,g)(N)=0$ for any Borel subset $N$ of $X$ of capacity zero by
		\cite[Lemma 3.2.4]{FOT}.
}\end{definition}
By choosing a function $f \in \mathcal{F}$ such that  $f \equiv 1$ in a neighborhood of $B(x,R+r)$ in the above definition, we note that \hyperlink{Gcap}{$\operatorname{Gcap}(\Psi)$} implies  \hyperlink{cap}{$\operatorname{cap}(\Psi)_\le$}.

We recall the definition of harmonic functions and the elliptic Harnack inequality.

\begin{definition}[Harmonic functions and elliptic Harnack inequality] {\rm 
		Let $(X,d,m,\mathcal{E},\mathcal{F})$ be a MMD space. A function $h \in \mathcal{F}$
		is said to be \emph{$\mathcal{E}$-harmonic} on an open subset $U$ of $X$, if
		\begin{equation}\label{e:harmonic}
		\mathcal{E}(h,f)=0,
		\end{equation}
		for all $f\in\mathcal{F}\cap\mathcal{C}_{\mathrm{c}}(X)$ with $\operatorname{supp}[f]\subset U$, where $\operatorname{supp}[f]$ denotes the support of $f$.
		
		We say that  a MMD space $(X,d,m,\mathcal{E},\mathcal{F})$ satisfies the \emph{elliptic Harnack inequality} (abbreviated as EHI), if there exist $C >1, \delta \in (0,1)$ such that for all $x \in X, r>0$ and for any non-negative harmonic function $h$ on the ball $B(x,r)$, we have
		\begin{equation} \label{EHI} \tag*{$\operatorname{EHI}$}
		\esssup_{B(x,\delta r)} h \le C \essinf_{B(x,\delta r)} h.
		\end{equation}
	}
\end{definition}
We recall the definition of a remote ball.
\begin{definition}[Remote ball]{\rm
		Fix $o \in X$ and $\epsilon \in (0,1]$. We say that a ball $B(x,r)$ is \emph{$\epsilon$-remote} with respect to $o$, if $r \ge \frac 1 2 \epsilon d(o,x)$.
	}
\end{definition}
As an application of the annular connectivity result in Theorem \ref{t:annuluscurve}, we obtain the following stability of elliptic Harnack inequality under perturbation by admissible weights.
\begin{theorem}  \label{t:weight}
	Let $(X,d,m,\mathcal{E},\mathcal{F})$ be an unbounded MMD space that satisfies  \hyperlink{vd}{$\operatorname{(VD)}$}, \hyperlink{pi}{$\operatorname{PI}(\Psi)$}, \hyperlink{Gcap}{$\operatorname{Gcap}(\Psi)$} and \hyperlink{fvg}{$\textup{FVG}(\Psi)$}, where $\Psi$ satisfies \eqref{e:reg}. Let $w \in \mathcal{C}(X), w: X \to (0,\infty)$ be an admissible weight. Then the corresponding weighted MMD space $(X,d,w\,dm,\mathcal{E}_w,\mathcal{F}_w)$ satisfies the elliptic Harnack inequality.
\end{theorem}
\begin{proof} 
	By Corollary \ref{c:curve}, $B(x,r) \setminus B(x,r/2)\neq \emptyset$ for any ball $B(x,r)$. This along with \hyperlink{vd}{$\operatorname{(VD)}$} implies the following reverse volume doubling property: there exist $C_1,\alpha>0$ such that
	\[
	\frac{m(B(x,R))}{m(B(x,r))} \ge C_1^{-1} \left(\frac R r\right)^\alpha, \quad\mbox{for any $x \in X, 0<r \le R$.}
	\]
	Let $o \in X$ be the point as in Definition \ref{d:admiss} and let $\mu = w \,dm$ denote the weighted measure. 
	By \eqref{e:adm}, the weight $w$ is comparable to a constant on any $\epsilon$-remote ball with respect to $o$ for any $\epsilon \in (0,1]$. By choosing $\epsilon \in (0,1]$ small enough, the MMD space  $(X,d,w\,dm,\mathcal{E}_w,\mathcal{F}_w)$ satisfies the volume doubling property, reverse volume doubling property, \hyperlink{Gcap}{$\operatorname{Gcap}(\Psi)$} and \hyperlink{fvg}{$\textup{FVG}(\Psi)$} for all remote balls. 
	By \cite[proof of Theorem 1.2]{GHL}, the MMD space  $(X,d,w\,dm,\mathcal{E}_w,\mathcal{F}_w)$ satisfies \hyperlink{ehi}{$\operatorname{EHI}$} for all remote balls.
	By applying \cite[Lemma 6.4]{GS2} for the metric measure space $(X,d,m)$, we obtain the annuli covering condition in \cite[Definition 6.2]{GS2}. 
	By Theorem \ref{t:annuluscurve} and \cite[Lemma 6.3]{GS2}, we obtain that the the MMD space  $(X,d,w\,dm,\mathcal{E}_w,\mathcal{F}_w)$ satisfies \hyperlink{ehi}{$\operatorname{EHI}$} for all  balls.
\end{proof}
\begin{remark}
	We remark that the \hyperlink{fvg}{$\textup{FVG}(\Psi)$} condition in Theorem \ref{t:weight} is necessary. For example, for Brownian motion on $\mathbb{R}$, the weight $w(x)= (1+x^2)^{\alpha/2}$ where $\alpha >1$ fails to satisfy the Liouville property and hence the elliptic Harnack inequality. This is because, the diffusion corresponding to the Dirichlet form $(\mathcal{E}_w,\mathcal{F}_w)$ has two transient ends at $\pm \infty$. Hence, the probability that the diffusion eventually ends up in one of them (say $+\infty$) is a non-constant positive harmonic function.
\end{remark}

\bibliographystyle{amsalpha}

\end{document}